\documentclass[12pt]{article}
\usepackage{amsmath,amssymb}
\usepackage{amscd}
\usepackage{comment}
\usepackage{color}
\usepackage{enumerate}
\usepackage{cases}
\usepackage{mathtools}
\usepackage[all]{xy}
\usepackage{amsthm}
\usepackage{pb-diagram}
\theoremstyle{definition}
\newtheorem{defi}{Definition}
\newtheorem{theo}[defi]{Theorem}
\newtheorem{lemm}[defi]{Lemma}
\newtheorem{prop}[defi]{Proposition}

\newtheorem{rem}[defi]{Remark}

\def\det{{\rm det}}

\def\End{{\rm End}}

\def\det{{\rm det}}

\def\Cap{{\rm Cap}}
\def\supp{{\rm supp}}

\def\R{{\mathbb R}}
\def\Z{{\mathbb Z}}
\def\C{{\mathbb C}}
\def\N{{\mathbb N}}

\def\CC{{\cal C}}

\def\inum{{\sqrt{-1}}}

\begin{document}

\title {Cyclic Higgs bundles, subharmonic functions, and the Dirichlet problem}
\author {Natsuo Miyatake}
\date{}
\maketitle
\begin{abstract} We demonstrate the existence and uniqueness of the solution to the Dirichlet problem for a generalization of Hitchin's equation for diagonal harmonic metrics on cyclic Higgs bundles. The generalized equations are formulated using subharmonic functions. In this generalization, the coefficient exhibits worse regularity than that in the original equation. 
\end{abstract}

\section{Introduction}\label{section 1}
Let $X$ be a connected, possibly non-compact Riemann surface equipped with a K\"ahler metric $g_X$. We denote by $h_X$ the Hermitian metric on the canonical bundle $K_X\rightarrow X$, by $\omega_X$ the K\"ahler form, and by $\Lambda_{\omega_X}$ the adjoint of $\omega_X\wedge$. We choose a square root $K_X^{1/2}$ of the canonical bundle $K_X$. We define a vector bundle $E$ of rank $r$ as $E\coloneqq K_X^{(r-1)/2}\oplus K_X^{(r-3)/2}\oplus\cdots \oplus K_X^{-(r-3)/2}\oplus K_X^{-(r-1)/2}$. Let $h=(h_1,\dots, h_r)$ be a smooth diagonal Hermitian metric on $E$ with curvature $F_h=(F_{h_1},\dots, F_{h_r})$. We assume that $\det(h)$ is flat. Let $H_j\coloneqq h_j^{-1}\otimes h_{j+1}\otimes h_X$ be a Hermitian metric on the trivial bundle for each $j=1,\dots, r-1$, and $H_r=h_1\otimes h_r^{-1}\otimes h_X$ a Hermitian metric on $K_X^r$. For each $j=1,\dots, r$, we denote by $F_{H_j}$ the curvature associated with the metric $H_j$. Let $\varphi: X\rightarrow [-\infty,\infty)$ be a quasi-subharmonic function, i.e., a locally integrable function that is locally a sum of a subharmonic function and a smooth function (cf. \cite{GZ1}). Note that we omit the function which is identically $-\infty$ from the definition of the quasi-subharmonic function. The quasi-subharmonic function $\varphi$ is said to be an {\it $F_{H_r}$-subharmonic function} if the following holds in the sense of the distribution (cf. \cite[Section 8]{GZ1}):
\begin{align*}
\inum\partial\bar{\partial}\varphi+\inum F_{H_r}\geq 0.
\end{align*}
As we can easily see from the definition, for each $N\in\Z_{\geq1}$ and each $q_N\in H^0((K_X^r)^N)$, $\frac{1}{N}\log|q_N|_{H_r}^2$ is an $F_{H_r}$-subharmonic function, where $|q_N|_{H_r}^2$ is a square of the norm of $q_N$ measured by $H_r$. We consider the following PDE on $X$ defined by using an $F_{H_r}$-subharmonic function $\varphi$:
\begin{align}
\Delta_{\omega_X}\xi+\sum_{j=1}^r4k_j^\prime e^{(v_j,\xi)}v_j=-2\inum \Lambda_{\omega_X}F_h, \label{HEeq0}
\end{align}
where each symbol is defined as follows:
\begin{itemize}
\item Let $V$ be a vector space defined as $V\coloneqq \{x=(x_1,\dots, x_r)\mid x_1+\cdots +x_r=0\}$, which is identified with the set of trace-free diagonal matrices of rank $r$. Then for each $j=1,\dots, r$, $v_j \in V$ is a vector defined as $v_j \coloneqq u_{j+1}-u_j$, where $u_1,\dots, u_r$ is the canonical basis on $\R^r$.
\item We denote by $(\cdot, \cdot)$ the standard inner product on $\R^r$.
\item We denote by $\xi:X\rightarrow V$ a $V$-valued function which is a solution of equation (\ref{HEeq0}) in some sense.
\item We define $k_1^\prime,\dots, k_r^\prime$ as $k_j^\prime\coloneqq |1|_{H_j}$ $(j=1,\dots, r-1)$ and $k_r^\prime\coloneqq e^\varphi$, where $1$ is the canonical section of the trivial bundle, and $|1|_{H_j}$ is the norm measured by $H_j$.
\item We denote by $\Delta_{\omega_X}=-2\inum\Lambda_{\omega_X}\partial \bar{\partial}$ the geometric Laplacian.
\end{itemize}
Suppose that $X$ is a non-compact Riemann surface. Let $f_X:X\rightarrow \R$ be a smooth strictly subharmonic function such that $\{x\in X\mid f_X(x)\leq c\}$ is a compact subset for each $c\in\R$. We take a $c\in\R$ and set
\begin{align*}
Y&\coloneqq \{x\in X\mid f_X(x)<c\}, \\
\partial Y&\coloneqq f_X^{-1}(c), \\
\overline{Y}&\coloneqq Y\cup \partial Y.
\end{align*}
Our main theorem is as follows:
\begin{theo}\label{main theorem 1} {\it For each $V$-valued continuous function $\eta=(\eta_1,\dots, \eta_r):\partial Y\rightarrow V$, there exists a $V$-valued function $\xi=(\xi_1,\dots, \xi_r):Y\rightarrow V$ that satisfies the following:
\begin{enumerate}[(a)]
\item \label{a} $\xi$ is a $C^{1,\alpha}$-function for any $\alpha\in(0,1)$ and solves equation (\ref{HEeq0}) in the sense of the distribution.
\item \label{b} The following boundary condition holds: 
\begin{align*}
\lim_{z\to\zeta}\xi(z)=\eta(\zeta) \ \text{for all $\zeta\in\partial Y$}.
\end{align*} 
Moreover, any $V$-valued function that satisfies conditions $(\ref{a})$ and $(\ref{b})$ is unique.
\end{enumerate}
}
\end{theo}
\begin{rem} Let $\xi:U\rightarrow V$ be a $V$-valued locally $L^\infty$-function defined on an open subset $U\subseteq X$ of $X$. We say that {\it $\xi$ solves equation (\ref{HEeq0}) in the sense of the distribution}, or that {\it $\xi$ is a weak solution to equation (\ref{HEeq0})} if the following holds:
\begin{align*}
&\int_U\{(\xi,\Delta_{\omega_X}\phi)+(\sum_{j=1}^r4k_j^\prime e^{(v_j,\xi)}v_j+2\inum \Lambda_{\omega_X}F_h,\phi)\}=0 \ \text{for all $\phi\in C^\infty_c(U,V)$}, 
\end{align*}
where we denote by $C^\infty_c(U,V)$ the space of all smooth $V$-valued functions with compact support, and the integral is taken with respect to the K\"ahler form $\omega_X$. Throughout the paper, we use the terms ``in the sense of the distribution" and ``weak solution" for equations or inequalities including the Laplace operator, not limited to equation (\ref{HEeq0}), in the sense described above. Note that for the Poisson equation, the definition of a weak solution is not unique (cf. \cite{FF1}). 
\end{rem}
\begin{rem}\label{potential function} We use the notion of the $F_{H_r}$-subharmonic function even when $X$ is a non-compact Riemann surface. Note that on such a surface, the $F_{H_r}$-subharmonic function can globally be expressed as the sum of a subharmonic function and a smooth function related to the curvature $F_{H_r}$. Specifically, if $X$ is not compact, we can take a global holomorphic frame $e:X\rightarrow K_X^r$ (cf. \cite[Section 30]{For1}) for the holomorphic line bundle $K_X^r\rightarrow X$. Then $\tilde{\varphi}\coloneqq \varphi-\log H_r(e,e)$ is a subharmonic function on $X$ since we have $\bar{\partial}\partial \log H_r(e,e)=F_{H_r}$.
\end{rem}
\begin{rem}\label{locally bounded} The coefficient $e^\varphi$ is a bounded function on any compact subset $K\subseteq X$. Although proving this assertion is straightforward even if $X$ itself is compact, but to avoid redundancy, we proceed by assuming that $X$ is a non-compact Riemann surface. Since $X$ is not compact, the $F_{H_r}$-subharmonic function $\varphi$ decomposes into the sum of a subharmonic function $\tilde{\varphi}$ and a smooth function $-\log H_r(e,e)$, as explained in Remark \ref{potential function}. This implies that $\varphi$ is an upper semicontinuous function and, therefore, attains its maximum on $K$ (see \cite[Chapter 2.1]{Ran1}). Consequently, the coefficient $e^\varphi$ is bounded on $K$. In particular, for every $V$-valued locally $L^\infty$-function $\xi:X\rightarrow V$, meaning that the function is a $V$-valued $L^\infty$-function over all compact subset $K\subseteq X$, the function $e^\varphi e^{(v_r,\xi)}$ is also a locally $L^\infty$-function. This ensures its well-definedness as a distribution.
\end{rem}
\begin{rem} Let $E^\vee$ be the dual vector bundle of $E$. The vector bundle $E$ is equipped with an isomorphism $S_E:E\rightarrow E^\vee$ defined as follows (cf. \cite{Hit1, LM1, LM3}):
\begin{align*}
S_E\coloneqq \left(
\begin{array}{ccc}
& &1\\
&\reflectbox{$\ddots$} &\\
1&&
\end{array}
\right): E\rightarrow E^{\vee}.
\end{align*}
A Hermitian metric $h_E$ on $E$ is said to be {\it real} (cf. \cite{Hit1, LM1,LM3}) if the above $S_E$ is isometric with respect to $h_E$ and $h_E^\vee$, where $h_E^\vee$ is the natural Hermitian metric on $E^\vee$ induced from $h_E$. From the arguments based on the uniqueness of the solution (cf. \cite[Section 7]{Hit1}, \cite[Corollary 3.24]{LM1}, \cite[Section 2.3.5]{LM3}), we can show that if the metric $(e^{\eta_1}h_1\left.\right|_{\partial Y},\dots, e^{\eta_r}h_r\left.\right|_{\partial Y})$ on the boundary is real, then the metric $(e^{\xi_1}h_1\left.\right|_Y,\dots, e^{\xi_r}h_r\left.\right|_Y)$ induced from the solution $\xi=(\xi_1,\dots,\xi_r)$ of the Dirichlet problem in Theorem \ref{main theorem 1} is also real.
\end{rem}

Equation (\ref{HEeq0}) is a generalization of Hitchin's equation \cite{Hit0} for diagonal harmonic metrics on cyclic Higgs bundles \cite{Bar1, DL1}, which was introduced in \cite[Example 1]{Miy2}. It should be noted that in \cite[Example 1]{Miy2}, only the case where $Y$ is a domain of $\C$ and $h$ is the metric induced by the standard metric on $\C$ is discussed. In Section \ref{section 2}, we explain the motivation behind introducing equation (\ref{HEeq0}) and solving its corresponding Dirichlet problem. In Section \ref{section 3}, we establish some fundamental a priori estimates to the solution of equation (\ref{HEeq0}) by slightly modifying the proofs of \cite[Theorem 2 and Theorem 3]{Miy2}. In Section \ref{section 4}, we give a proof of Theorem \ref{main theorem 1}. 
\section{Cyclic Higgs bundles with multi-valued Higgs fields}\label{section 2}
We first briefly recall the definition of cyclic Higgs bundles. Let $X$ be a connected Riemann surface. We carry over the symbols used in Section 1. We take a $q\in H^0(K_X^r)$. For each $j=1,\dots, r-1$, we set $\Phi(q)_{j+1,j}=1$ and $\Phi(q)_{1,r}=q$. We define $\Phi(q)\in H^0(\End E\otimes K_X)$ as $\Phi(q)\coloneqq \sum_{j=1}^{r-1}\Phi(q)_{j+1,j}+\Phi(q)_{1,r}$, where $\Phi(q)_{i,j}$ is considered to be the $(i,j)$-component of $\Phi(q)$, and $1$ (resp. $q$) is considered to be a $K_X^{-1}$ (resp. $K_X^{r-1}$)-valued holomorphic 1-form. We call $(E,\Phi(q))$ a cyclic Higgs bundle (cf. \cite{Bar1, DL1, LM1, LM2}). Cyclic Higgs bundles are examples of the cyclotomic Higgs bundles which were introduced in \cite{Sim3}. We set $k_1,\dots, k_r$ as $k_j=k_j^\prime=|1|_{H_j}$ $(j=1,\dots, r-1)$, $k_r=|q|_{H_r}$. The corresponding Hitchin's equation \cite{Hit0} for a diagonal harmonic metric $(e^{f_1}h_1,\dots, e^{f_r}h_r)$ is then given by:
\begin{align}
\Delta_{\omega_X}\xi+\sum_{j=1}^r4k_je^{(v_j,\xi)}v_j=-2\inum \Lambda_{\omega_X}F_{h}, \label{HEeq}
\end{align}
where $\xi$ is defined as $\xi\coloneqq (f_1,\dots, f_r)$. Equation (\ref{HEeq}) is also called Toda lattice with opposite sign (see \cite{GL1}). As we noted in Section 1, $\log |q|_{H_r}$ is an $F_{H_r}$-subharmonic function, and thus, equation (\ref{HEeq}) is a special case of equation (\ref{HEeq0}) if we impose the condition $f_1+\cdots +f_r=0$ on $f_1,\dots, f_r$.

Let $N\in \Z_{\geq 2}$ and $q_N\in H^0((K_X^r)^N)$. We next consider a cyclic Higgs bundle $(E, \Phi(q_N^{1/N}))$ with the following multi-valued Higgs field: 
\begin{align*}
\Phi(q_N^{1/N})\coloneqq 
\left(
\begin{array}{cccc}
0 & && q_N^{1/N}\\
1 & \ddots && \\
&\ddots&\ddots& \\
&&1&0
\end{array}
\right).
\end{align*}
It can be observed that Hitchin's equation for diagonal harmonic metrics on cyclic Higgs bundles depends only on the absolute value of $q$. Therefore, although the Higgs field $\Phi(q_N^{1/N})$ is multi-valued, Hitchin's equation for a diagonal harmonic metric on $(E,\Phi(q_N^{1/N}))$ is well-defined, and the equation for diagonal harmonic metrics on a cyclic Higgs bundle with a multi-valued Higgs field $(E,\Phi(q_N^{1/N}))$ coincides with equation (\ref{HEeq0}) with an $F_{H_r}$-subharmonic function $\frac{1}{N}\log|q_N|_{H_r}^2$. Let $h_N$ be a solution to Hitchin's equation for a cyclic Higgs bundle $(E,\Phi(q_N^{1/N}))$ with a multi-valued Higgs field $\Phi(q_N^{1/N})$. If we choose a well-defined local section $q_N^{1/N}$ on an open subset $U\subseteq X$, then $(E,\Phi(q_N^{1/N}), h_N)$ is a harmonic bundle on $U$. Alternatively, if we choose a ramified covering $\pi: Z_N\rightarrow X$ where $q_N^{1/N}$ is a well-defined section of $\pi^\ast (K_X^{r-1})\otimes K_{Z_N}\rightarrow Z_N$, then the triplet $(\pi^\ast E, \pi^\ast \Phi(q_N^{1/N}), \pi^\ast h_N)$ becomes a harmonic bundle over $Z_N$. 

As we noted above, Hitchin's equation for a diagonal harmonic metric on $(E,\Phi(q_N^{1/N}))$ is well-defined, although the Higgs field $\Phi(q_N^{1/N})$ is multi-valued. Additionally, it is evident that any functions constructed from a solution $h_N$ to the Hitchin's equation will be well-defined over $X$, if they depend solely on the absolute value of $q_N$ . For example, the norm $|\Phi(q_N^{1/N})|^2_{h_N,h_X}$ and the bracket $\Lambda_{\omega_X}[\Phi(q_N^{1/N})\wedge\Phi(q_N^{1/N})^{\ast h_N}]$ of the Higgs field $\Phi(q_N^{1/N})$ are well-defined functions. Furthermore, obviously, the amounts that can be written as a constant multiple or combination of them, such as energy density $e(h_N)$ of harmonic maps and the sectional curvature $\kappa(h_N)$ of the image of harmonic maps (cf. \cite{DL1, Li1}), are also well-defined functions on $X$.

In this paper, we pose the problem of considering what might happen when $N$ approaches infinity. More specifically, we pose the problem to consider the asymptotic behavior of a sequence of Hermitian metrics $(h_N)_{N\in \mathbb{N}}$ such that for each $N$, $h_N$ is a diagonal solution to Hitchin's equation for the cyclic Higgs bundle $(E,\Phi(q_N^{1/N}))$ with a multi-valued Higgs field when $N$ tends to infinity. To specify what we mean by ``consider the asymptotic behavior", we pose the following specific questions:
\begin{itemize}
\item Does a sequence $(h_N)_{N\in\N}$ converge a Hermitian metric in some topology? Also, how fast will it converge?
\item How does sequences of well-defined functions, such as $(e(h_N))_{N\in\N}, (\kappa(h_N))_{N\in\N}, \dots $ behave as $N$ approaches infinity? For example, is it possible to choose a non-trivial sequence $(h_N)_{N\in\N}$ so that the averages $\int_X e(h_N), \int_X \kappa(h_N),...$ monotonically decreases or increases? How fast will they decay or increase? Also, what if we looked at the behavior at each point of $X$ rather than the average?
\item Is it possible to evaluate the value of the average of well-defined functions such as $\int_Xe(h_N), \int_X \kappa(h_N),...$ in the limit when $N$ tends to infinity? Also, what if we looked at the behavior at each point of $X$ rather than the average?
\item The difference of Hermitian metrics $h_{N,j}^{-1}\otimes h_{N, j+1}$ $(j=1,\dots, r-1)$ defines a metric $d_{N,j}$ on $X$ (see \cite{LM1}), where we denote by $h_{N,j}$ the $j$-th component of the Hermitian metric: $h_N=(h_{N,1},\dots, h_{N,r})$. Moreover, the Hermitian metric $h_{N,r}^{-1}\otimes h_{N,1}$ coupled with $q_N^{1/N}$ defines a degenerate metric $d_{N,r}$ on $X$ (see \cite{LM1}). How does the sequence of metric spaces $(X, d_{N, 1},\dots, d_{N, r-1}, d_{N,r})_{N\in\N}$ behave? For example, is the completeness of the metrics (cf. \cite{LM1}) preserved in the limit? 
\end{itemize}

Let $SH(X, F_{H_r})$ be the set of all $F_{H_r}$-subharmonic functions. As we noted in Section 1, for each positive integer $N$ and each $q_N\in H^0((K^r_X)^N)$, $\frac{1}{N}\log|q_N|_{H_r}^2$ is $F_{H_r}$-subharmonic, and moreover, all elements of such form are dense in $SH(X, F_{H_r})$ with respect to the $L^1_{loc}$-topology (see \cite{GZ1}). Therefore, any $F_{H_r}$-subharmonic function $\varphi\in SH(X, F_{H_r})$ is a limit of a sequence of $(\frac{1}{N}\log|q_N|^2_{H_r})_{N\in\N}$, where $q_N\in H^0((K_X^r)^N)$, at least for $L^1_{loc}$-topology. In considering the above problems, we introduce equation (\ref{HEeq0}) as a limit of Hitchin's equation for diagonal harmonic metrics on cyclic Higgs bundles $(E,\Phi(q_N^{1/N}))$ with a multi-valued Higgs field when $N$ tends to infinity.

If the coefficient $e^\varphi$ is smooth, we can directly apply methods from \cite{Don0, Don1, Sim1} either to find a solution for equation (\ref{HEeq0}) or for its evolution equation (cf. \cite{Miy2}, to which we refer the reader for further explanations). More specifically, we can construct a time-global solution to the evolution equation of (\ref{HEeq0}) on a compact manifold with a possibly empty boundary with the Dirichlet boundary condition, by using the techniques of the evolution equation of Hermitian-Einstein equation \cite{Don0, Sim1}. If the manifold has a non-empty boundary, we can demonstrate, by using the Donaldson's argument \cite{Don1}, that the time global solution to the evolution equation converges to a solution of equation (\ref{HEeq0}) that satisfies a Dirichlet boundary condition. For compact manifolds without boundary by using the functional of equation (\ref{HEeq0}) (cf. \cite{Miy1}) we can show the convergence of the time global solution of the evolution equation to a solution of equation (\ref{HEeq0}). We can also solve equation (\ref{HEeq0}) by directly applying \cite[Theorem 1]{Miy1} when the coefficient $e^\varphi$ is smooth. In this paper, we extend the Dirichlet problem to a more general case. 
The Dirichlet problem for the Hermitian-Einstein equation was first solved in \cite{Don1}. This theorem holds significant utility, particularly when constructing a global solution to the Hermitian-Einstein equation on non-compact manifolds with compact exhaustions \cite{LM1, LM2, Moc1}. It is also worth noting that the Dirichlet problem for elliptic equations has been studied over the years, with specific emphasis on its link to potential theory (cf. \cite{GZ1, Ran1}). While there are multiple aspects to investigate in equation (\ref{HEeq0}), our paper primarily focuses on the Dirichlet problem, considering its distinctive significance.

\begin{rem} 
The aforementioned problem focuses on the increasing number of zeros in the holomorphic $r$-differential. Conversely, one could conceive a dual problem that considers the consequences of an increasing number of poles in the $r$-differential. While this paper will not delve into this topic in depth, it will be discussed in a subsequent paper.
\end{rem}
\begin{rem} The notion of cyclic Higgs bundles can be generalized to the case where the vector bundle is of the form $E=L_1\oplus\cdots \oplus L_r$ with arbitrary holomorphic line bundles $L_1, \dots, L_r\rightarrow X$ (cf. \cite{DL1}). Equation (\ref{HEeq0}) can also be generalized to such a case, and Theorem \ref{main theorem 1} can be extended to the more generalized equation.
\end{rem}

\begin{rem}
If the purpose is simply to solve equation (\ref{HEeq0}), there is no need to explicitly state the existence and convergence of the time global solution of the evolution equation. However, the heat equation itself is very interesting, so we have explicitly written the existence and convergence of the global time solution as above. For example, an interesting question is, is it possible to construct a time global solution that is complete (cf. \cite{LM1}) at each instant?
\end{rem}
\begin{rem} The problems described above are both influenced by and motivated by the study of the asymptotic behavior of the complex polynomials and sections of holomorphic line bundles (see, e.g., \cite{Ran1, ST1, SZ1} and the references therein).
\end{rem}

\section{Fundamental a priori estimates}\label{section 3}
Before beginning the proof of Theorem \ref{main theorem 1}, we establish some fundamental a priori estimates for the solution of equation (\ref{HEeq0}) by slightly modifying the proofs of \cite[Theorem 2 and Theorem 3]{Miy2}. For the Hermitian-Einstein equation of Higgs bundles, the following estimates have been established in \cite[Lemma 3.1 and Lemma 10.1]{Sim1} (see also \cite{LM1, LM2, Sim2}). We will use Proposition \ref{prop2} below in Section \ref{section 4} to prove the uniqueness of the boundary value problem in Theorem \ref{main theorem 1}. We carry over the notation from Section \ref{section 1}. 
The following holds:
\begin{prop}\label{prop2}{\it Let $\xi=(\xi_1,\dots, \xi_r),\xi^\prime=(\xi_1^\prime,\dots,\xi_r^\prime):X\rightarrow V$ be locally $L^\infty$-functions. 
Suppose that there exist $V$-valued locally $L^1$-functions $\tilde{\xi},\tilde{\xi}^\prime: X\rightarrow V$ such that $\Delta_{\omega_X}\xi=\tilde{\xi}$ and $\Delta_{\omega_X}\xi^\prime=\tilde{\xi}^\prime$ as distributions on $X$. Then the following holds in the sense of the distribution:
\begin{align}
&\Delta_{\omega_X}\log|\sum_{j=1}^re^{(\xi_j-\xi^\prime_j)}| \notag \\
\leq &|\Delta_{\omega_X}\xi+\sum_{j=1}^r4k_j^\prime e^{(v_j,\xi)}v_j+2\inum \Lambda_{\omega_X} F_h|+|\Delta_{\omega_X}\xi^\prime +\sum_{j=1}^r4k_j^\prime e^{(v_j,\xi^\prime)}v_j+2\inum \Lambda_{\omega_X} F_h|. \label{LSE}
\end{align}
In particular, if $\xi$ and $\xi^\prime$ solve equation (\ref{HEeq0}) in the sense of the distribution, then we have 
\begin{align*}
\Delta_{\omega_X}\log|\sum_{j=1}^re^{(\xi_j-\xi^\prime_j)}| \leq 0.
\end{align*}
}
\end{prop}
\begin{prop}\label{prop3}{\it Suppose that $\xi=(\xi_1,\dots, \xi_r)$ is an $L^\infty$-weak solution to equation (\ref{HEeq0}). Then the following holds in the sense of the distribution:
\begin{align}
\Delta_{\omega_X}\log\bigl(\sum_{j=1}^r4k_j^\prime e^{(v_j,\xi)}\bigr)\leq -\frac{\left|\sum_{j=1}^r4k_j^\prime e^{(v_j,\xi)}v_j\right|^2}{\left|\sum_{j=1}^r4k_j^\prime e^{(v_j,\xi)}\right|}+2\inum\Lambda_{\omega_X} F_{h_X}, \label{LSE2}
\end{align}
where $F_{h_X}$ is the curvature of the metric $h_X$.
}
\end{prop}
\begin{rem}\label{very important remark} Let $f_1,\dots, f_r: X\rightarrow \R$ be locally integrable functions. Then the following inequality holds:
\begin{align*}
\max\{f_1,\dots, f_r\}\leq \log(\sum_{j=1}^re^{f_j})\leq \max\{f_1,\dots, f_r\}+\log(r).
\end{align*}
Consequently, $\log(\sum_{j=1}^re^{f_j})$ is also a locally integrable function. Therefore, the distribution $\Delta_{\omega_X}\log(\sum_{j=1}^re^{f_j})$ is well-defined. In particular, left-hand sides of inequality (\ref{LSE}) and (\ref{LSE2}) are both well-defined. Furthermore, as highlighted in Remark \ref{locally bounded}, the coefficient $e^\varphi$ is a locally bounded function, ensuring that the right-hand sides of (\ref{LSE}) and (\ref{LSE2}) are both unambiguously defined as distributions.
\end{rem}
While the following concept is generally familiar, we provide a specific definition for the sake of clarity:
\begin{defi}[cf. \cite{GZ1, Ran1}] \label{mollifier} 
Let $B(0,1)\coloneqq \{z\in\C\mid |z|<1\}$ be the unit open ball in the complex plane $\C$. We call a function $\chi:\C\rightarrow \R$ a {\it mollifier} if it satisfies the following conditions (cf. \cite[p.49, Theorem 2.7.2]{Ran1}):
\begin{align*}
\chi\in C^\infty(\C,\R), \ \chi\geq 0, \ \chi(z)=\chi(|z|), \ \supp\chi\subseteq B(0,1), \ \int_\C\chi=1,
\end{align*}
where we denote by $\supp\chi$ the support of the function $\chi$. Let $\chi$ be a mollifier. In the same way as \cite[p.49, Theorem 2.7.2]{Ran1}, for $\epsilon>0$ we define a function $\chi_\epsilon$ as follows: 
\begin{align*}
\chi_\epsilon(z)\coloneqq \frac{1}{\epsilon^2}\chi\left(\frac{z}{\epsilon}\right)\ \text{for $z\in\C$}.
\end{align*}
\end{defi}
\begin{rem} The notion of a mollifier is usually defined for a class of functions broader than the one above.
\end{rem}
We prepare the following lemma:
\begin{lemm}\label{lemma1}
{\it Let $G=(f_1,\dots, f_r):B(0,1)\rightarrow \R^r$ be an $\R^r$-valued locally $L^1$-function with respect to the Euclidean metric. We set $b_1\coloneqq \sum_{j=1}^re^{f_j/2}u_j, b_2\coloneqq \sum_{j=1}^re^{f_j}u_j$. We assume that there exists an $\R^r$-valued locally integrable function $\tilde{G}=(\tilde{f}_1,\dots, \tilde{f}_r)$ such that
\begin{align}
\Delta f_j\leq \tilde{f}_j \ \text{in the sense of the distribution for all $j=1,\dots, r$},
\end{align}
where we denote by $\Delta$ the Laplacian $\left(-4\frac{\partial^2}{\partial z\partial \bar{z}}\right)$ for the Euclidean metric. Then the following inequality holds in the sense of the distribution:
\begin{align}
\Delta \log|b_1|^2\leq \left(\tilde{G},b_2/|b_1|^2\right).\label{b2}
\end{align}
}
\end{lemm}
\begin{rem} The right hand side of inequality (\ref{b2}) is locally integrable since $b_2/|b_1|^2$ is a bounded function. 
\end{rem}
\begin{rem}
As we remarked in Remark \ref{very important remark}, $\log|b_1|^2$ is locally integrable, and thus the left-hand side of inequality (\ref{b2}) is well-defined as a distribution.
\end{rem}
\begin{proof}[Proof of Lemma\ref{lemma1}] 
Let $\chi:\C\rightarrow \R$ be a mollifier and let $(\chi_\epsilon)_{\epsilon>0}$ denote the associated family of functions (see Definition \ref{mollifier} above). We define $G_\epsilon=(f_{1,\epsilon},\dots,f_{r,\epsilon})$ and $\tilde{G}_\epsilon=(\tilde{f}_{1,\epsilon},\dots,\tilde{f}_{r,\epsilon})$ as the convolution 
\begin{align*}
&G_\epsilon\coloneqq G\ast\chi_\epsilon=(f_1,\ast\chi_\epsilon,\dots, f_r\ast\chi_\epsilon), \\
&\tilde{G}_\epsilon\coloneqq \tilde{G}\ast\chi_\epsilon=(\tilde{f}_1,\ast\chi_\epsilon,\dots, \tilde{f}_r\ast\chi_\epsilon).
\end{align*}
These are defined on the open ball $B(0,1-\epsilon)\coloneqq\{z\in\C\mid|z|<1-\epsilon\}$ of $B(0,1)$ (see \cite[Definition 2.7.1]{Ran1}). 
We set $b_{1,\epsilon}$ and $b_{2,\epsilon}$ as follows:
\begin{align*}
&b_{1,\epsilon}\coloneqq \sum_{j=1}^re^{f_{j,\epsilon}/2}u_j, \\
&b_{2,\epsilon}\coloneqq \sum_{j=1}^re^{f_{j,\epsilon}}u_j. 
\end{align*}
From \cite[Proof of Theorem 2]{Miy2}, the following inequality holds for $G_\epsilon$, $b_{1,\epsilon}$ and $b_{2,\epsilon}$:
\begin{align}
\Delta\log|b_{1,\epsilon}|^2\leq \left(\Delta G_\epsilon,b_{2,\epsilon}/|b_{1,\epsilon}|^2\right).\label{b3}
\end{align}
Therefore, inequality (\ref{b2}) holds for $\tilde{G}_\epsilon$, $b_{1,\epsilon}$, and $b_{2,\epsilon}$:
\begin{align}
\Delta\log|b_{1,\epsilon}|^2\leq \left(\tilde{G}_\epsilon,b_{2,\epsilon}/|b_{1,\epsilon}|^2\right). \label{b4}
\end{align}
We then show that as $\epsilon\rightarrow 0$, $\Delta\log|b_{1,\epsilon}|^2$ and $\left(\tilde{G}_\epsilon,b_{2,\epsilon}/|b_{1,\epsilon}|^2\right)$ converge weakly to $\Delta \log|b_1|^2$ and $\left(\tilde{G},b_2/|b_1|^2\right)$ in the sense of distributions, respectively. Let $\phi:B(0,1)\rightarrow \R$ be a smooth function with a compact support. Due to the property of the mollifier, as $\epsilon\to0$, $G_\epsilon$ (resp. $\tilde{G}_\epsilon$) converges strongly to $G$ (resp. $\tilde{G}$) in the $L^1$-topology on each compact subset of $B(0,1)$. Specifically, this means $G_\epsilon$ (resp. $\tilde{G}_\epsilon$) converges to $G$ (resp. $\tilde{G}$) almost everywhere on each compact subset. Thus, by the Lebesgue's dominated convergence theorem, we find that
\begin{align*}
\int_{B(0,1)}\log|b_{1,\epsilon}|^2\Delta\phi\rightarrow \int_{B(0,1)}\log|b_1|^2\Delta\phi
\end{align*}
and 
\begin{align*}
\int_{B(0,1)}\left(\tilde{G}_\epsilon,b_{2,\epsilon}/|b_{1,\epsilon}|^2\right)\phi\rightarrow\int_{B(0,1)}\left(\tilde{G},b_2/|b_1|^2\right)\phi
\end{align*}
as $\epsilon\to0$, respectively. This establishes the desired claim.
\end{proof}

\begin{proof}[Proof of Proposition \ref{prop2}] It is enough to consider the case where $X$ is the open ball $B(0,1)$. Also, we can assume that the K\"ahler metric is the Euclidean metric. We define $\R^r$-valued functions $G$ and $\tilde{G}$ as follows:
\begin{align*}
&G\coloneqq \xi-\xi^\prime, \\
&\tilde{G}\coloneqq \tilde{\xi}-\tilde{\xi}^\prime. 
\end{align*}
Then by applying Lemma \ref{lemma1} to the above $G$ and $\tilde{G}$, and by combining the calculation in \cite[Proof of Theorem 2]{Miy2}, we have the desired inequality.
\end{proof}
\begin{proof}[Proof of Proposition \ref{prop3}] We set $I(\xi):X\rightarrow V$ as follows:
\begin{align*}
I(\xi)\coloneqq -\sum_{j=1}^r4k_je^{(v_j,\xi)}v_j-2\inum\Lambda_{\omega_X}F_h.
\end{align*}
We define $\R^r$-valued functions $G$ and $\tilde{G}$ as follows:
\begin{align*}
&G\coloneqq ((\xi,v_1)+\log(4k_1^\prime),\dots, (\xi,v_r)+\log(4k_r^\prime)), \\
&\tilde{G}\coloneqq ((I(\xi),v_1)+2\inum\Lambda_{\omega_X}F_{H_1}, \dots, (I(\xi),v_r)+2\inum\Lambda_{\omega_X}F_{H_r}).
\end{align*}
Then it can be verified that the following holds in the sense of the distribution:
\begin{align}
\Delta_{\omega_X}((\xi,v_j)+\log(4k_j^\prime))\leq (I(\xi),v_j)+2\inum\Lambda_{\omega_X}F_{H_j} \ \text{for $j=1,\dots, r$.} \label{important inequality}
\end{align}
As in the proof of Proposition \ref{prop3}, it can be supposed that $X$ is the open ball $B(0,1)$, and the K\"ahler metric is the Euclidean metric. From (\ref{important inequality}), we can apply Lemma \ref{lemma1} to the above $G$ and $\tilde{G}$. Then by the same calculation as in \cite[Proof of Theorem 3]{Miy2}, we have the desired inequality.
\end{proof}

\section{Proof of Theorem \ref{main theorem 1}}\label{section 4}
In order to prove the existence of a solution to equation (\ref{HEeq0}), we adopt the method of using the Schauder fixed point theorem (cf. \cite[Section 5.4.4]{GZ1} and \cite[p.143, Theorem 5.28]{Rud1}). Let $L^\infty(Y,V)$ denote the set of all $V$-valued $L^\infty$-functions over $Y$. To avoid potential misunderstandings, we use $L^\infty(Y,V)$ to denote the set of all $V$-valued $L^\infty$-functions, rather than as the set of equivalence classes. We first prove the following lemma:
\begin{lemm}\label{important lemma} There exist $\xi^{(-)}=(\xi^{(-)}_1,\dots, \xi^{(-)}_r),\ \xi^{(+)}=(\xi^{(+)}_1,\dots, \xi^{(+)}_r)\in L^\infty(Y,V)$ such that 
\begin{enumerate}[(i)]
\item \label{i} For each $j=1,\dots, r-1$, $\xi^{(-)}_j-\xi^{(+)}_j$ is a subharmonic function.
\item \label{ii} For each $j=1,\dots, r-1$, $\xi_j^{(-)}$ and $\xi_j^{(+)}$ are quasi-subharmonic functions.
\item \label{iii}The following holds in the sense of the distribution:
\begin{align}
&\Delta_{\omega_X}\xi_1^{(-)}\leq -4k_r^\prime e^{\xi^{(+)}_1-\xi_{r}^{(+)}}-2\inum\Lambda_{\omega_X}F_{h_1}, \label{iv0} \\
&\Delta_{\omega_X}\xi_j^{(-)}\leq -4k_{j-1}^\prime e^{\xi^{(+)}_j-\xi_{j-1}^{(-)}}-2\inum\Lambda_{\omega_X}F_{h_j} \ \text{for $j=2,\dots, r-1$},\label{iv1} \\
&\Delta_{\omega_X}\xi_j^{(+)}\geq 4k_j^\prime e^{\xi_{j+1}^{(+)}-\xi_j^{(-)}}-2\inum\Lambda_{\omega_X}F_{h_j}\ \text{for $j=1,\dots, r-1$}. \label{iv2} 
\end{align}
\item \label{iv}It holds that $\xi^{(-)}_j\leq\xi^{(+)}_j$ for all $j=1,\dots, r-1$.
\item \label{v} $\xi^{(-)}$ and $\xi^{(+)}$ satisfy the following boundary conditions:
\begin{align*}
&\lim_{z\to\zeta}\xi^{(-)}(z) =\ \eta(\zeta) \ \text{for all $\zeta\in\partial Y$}, \\
&\lim_{z\to\zeta}\xi^{(+)}(z) =\ \eta(\zeta) \ \text{for all $\zeta\in\partial Y$}.
\end{align*}
\end{enumerate}
\end{lemm}
\begin{proof} We first note that in Lemma \ref{important lemma}, it does not matter what the initial metric $h=(h_1,\dots, h_r)$ is. If we can prove the above lemma for some initial metric $h=(h_1,\dots, h_r)$ and arbitrary $\eta$, then by appropriately transforming $\xi^{(+)}$, $\xi^{(-)}$, and $\eta$ we can easily prove the above proposition for any initial metric. Therefore, for simplicity, from the beginning, we assume that the curvature $F_h$ of the initial metric $h$ is zero. Let $\phi_1,\dots, \phi_r$ be harmonic functions over $Y$ such that for each $j=1,\dots, r$, it holds that $\lim_{z\to\zeta}\phi_j(z)=\eta_j(\zeta)$ for all $\zeta\in\partial Y$ (cf. \cite{GZ1, Ran1}). For an $L^\infty$-subharmonic function $\rho: Y\rightarrow [-\infty, \infty)$ satisfying $\lim_{z\to\zeta}\rho(\zeta)=0$ for all $\zeta\in\partial Y$, we set 
\begin{align}
&\xi_j^{(-)}=\rho+\phi_j, \label{jminus}\\
&\xi_j^{(+)}=-\rho+\phi_j, \label{jplus}
\end{align}
for $j=1,\dots, r-1$. Then we define
\begin{align}
&\xi^{(-)}=(\xi_1^{(-)},\dots, \xi_{r-1}^{(-)}, -(\xi_1^{(-)}+\cdots +\xi_{r-1}^{(-)})), \label{minus}\\ 
&\xi^{(+)}=(\xi_1^{(+)},\dots, \xi_{r-1}^{(+)}, -(\xi_1^{(+)}+\cdots +\xi_{r-1}^{(+)})). \label{plus}
\end{align}
For the above $\xi^{(-)}$ and $\xi^{(+)}$, it can easily be checked that condition (\ref{i}), (\ref{ii}), and (\ref{v}) are satisfied. We can also check that condition (\ref{iv}) is satisfied since a subharmonic function $\rho$ satisfying $\lim_{z\to\zeta}\rho(\zeta)=0$ for all $\zeta\in\partial Y$ is a non-positive function (see \cite[p.12, Corollary 1.17]{GZ1}). We shall choose such a $\rho$ appropriately so that condition (\ref{iii}) is satisfied for $\xi^{(-)}$ and $\xi^{(+)}$ defined as above. We set 
\begin{align*}
&f_1\coloneqq \min{\{-4k_{j-1}e^{\phi_{j}-\phi_{j-1}}\mid j=1,\dots,r\}}, \\
&f_2\coloneqq \min{\{-4k_je^{\phi_{j+1}-\phi_j}\mid j=1,\dots,r\}}, \\
&f\coloneqq \min\{f_1,f_2\},
\end{align*}
where in the definition of $f_1$, $k_0$ and $\phi_0$ are interpreted as $k_r$ and $\phi_r$, respectively. Let $\rho$ be the unique $L^\infty$-subharmonic function that solves the following non-linear elliptic boundary problem:
\begin{align}
&-\Delta_{\omega_X}\rho=-fe^{-r\rho} \ \text{in the sense of the distribution}, \label{KW}\\
&\lim_{z\to\zeta}\rho(\zeta)=0 \ \text{for all $\zeta\in\partial Y$}.
\end{align}
The existence and the uniqueness of the solution $\rho$ to the above boundary value problem is guaranteed by \cite[p.154, Theorem 5.24]{GZ1} with a trivial variable transformation that replaces $r\rho$ with another variable. We will now demonstrate that for such a $\rho$, $\xi^{(-)}$ and $\xi^{(+)}$ defined as in (\ref{jminus}), (\ref{jplus}), (\ref{minus}), and (\ref{plus}) satisfy (\ref{iii}) of Lemma \ref{important lemma}. 
For $\xi^{(-)}$ and $\xi^{(+)}$ defined as (\ref{jminus}), (\ref{jplus}), (\ref{minus}), and (\ref{plus}), condition (\ref{iv0}), (\ref{iv1}), and (\ref{iv2}) can be rewritten as follows:
\begin{align}
&\Delta_{\omega_X}\rho\leq -4k_re^{-r\rho+\phi_1-\phi_r}, \label{rho1} \\
&\Delta_{\omega_X}\rho\leq-4k_{j-1}e^{-2\rho+\phi_j-\phi_{j-1}} \ \text{for $j=2,\dots, r-1$}, \label{rho2} \\
&\Delta_{\omega_X}\rho\leq -4k_je^{-2\rho+\phi_{j+1}-\phi_j} \ \text{for $j=1,\dots, r-1$}. \label{rho3}
\end{align}
Since the subharmonic function $\rho$ is a non-positive function, the following conditions (\ref{rho4}) and (\ref{rho5}) are stronger than the above (\ref{rho2}) and (\ref{rho3}):
\begin{align}
&\Delta_{\omega_X}\rho\leq-4k_{j-1}e^{-r\rho+\phi_j-\phi_{j-1}} \ \text{for $j=1,\dots, r-1$}, \label{rho4} \\
&\Delta_{\omega_X}\rho\leq -4k_je^{-r\rho+\phi_{j+1}-\phi_j} \ \text{for $j=1,\dots, r-1$}. \label{rho5}
\end{align}
It is easy to see that a subharmonic function $\rho$ which is a solution to the elliptic equation (\ref{KW}) satisfies (\ref{rho1}), (\ref{rho4}) and (\ref{rho5}). Then we have the desired claim.
\end{proof}
\begin{rem}\label{announcement} From conditions (\ref{ii}) and (\ref{iii}) in Lemma \ref{important lemma}, for each $j=1,\dots, r-1$, $\xi^{(-)}_j$ is a $-F_{h_j}$-subharmonic function, and $-\xi_j^{(+)}$ is an $F_{h_j}$-subharmonic function.
\end{rem}
\begin{rem} From the proof of Lemma \ref{important lemma}, we can construct $\xi^{(-)}$ and $\xi^{(+)}$ so that the following conditions hold in addition to (\ref{i}), (\ref{ii}), (\ref{iii}), (\ref{iv}), and (\ref{v}) in Lemma \ref{lemma1}:
\begin{align}
\Delta_{\omega_X}\xi_j^{(-)}\geq -C, \ \Delta_{\omega_X}\xi_j^{(+)} \leq C\ \text{for $j=1,\dots, r-1$} \label{iv3}
\end{align}
in the sense of the distribution, with $C$ being some positive constant. Indeed, for the $\xi^{(-)}$ and $\xi^{(+)}$ constructed in the proof of Lemma \ref{important lemma}, the following inequality holds for each $j=1,\dots, r-1$:
\begin{align*}
\Delta_{\omega_X}\xi_j^{(-)}&=\Delta_{\omega_X}\rho =fe^{-r\rho} \geq f, \\
\Delta_{\omega_X}\xi_j^{(+)}&=-\Delta_{\omega_X}\rho= -fe^{-r\rho}\leq -f.
\end{align*}
Since $f$ is an $L^\infty$-function, we have (\ref{iv3}) for the $\xi^{(-)}$ and $\xi^{(+)}$. If we impose condition (\ref{iv3}) to $\xi^{(-)}$ and $\xi^{(+)}$, then we can show that the sequence $(\xi_{(k)}^\prime)$ in Lemma \ref{L1} converges to a $\xi^\prime$ in capacity, as we will further discuss in Remark \ref{capacity}.
\end{rem}
We fix $\xi^{(-)}, \xi^{(+)}\in L^\infty(Y,V)$ satisfying the conditions of Lemma \ref{important lemma}. For a $V$-valued function $\xi=(\xi_1,\dots, \xi_r):Y\rightarrow V$, we consider the following condition:
\begin{enumerate}[($\ast$)]
\item For each $j=1,\dots, r-1$, $\xi_j^{(-)}-\xi_j$ and $\xi_j-\xi_j^{(+)}$ are subharmonic functions.
\end{enumerate}
We introduce the following notation:
\begin{defi} We denote by $SH(Y, \xi^{(-)},\xi^{(+)})$ the set of all $V$-valued functions that satisfy the above condition ($\ast$): 
\begin{align*}
SH(Y,\xi^{(-)},\xi^{(+)})\coloneqq \{\xi:Y\rightarrow V\mid \text{$\xi$ satisfies condition $(\ast)$}\}.
\end{align*}
\end{defi}
\begin{defi} For two vectors $v=(v_1,\dots, v_r),v^\prime=(v_1^\prime,\dots, v_r^\prime)\in V$, we denote by $v\leq v^\prime$ if $v_j\leq v_j^\prime$ holds for all $j=1,\dots, r-1$.
\end{defi}
\begin{defi} We define a set of $V$-valued functions $\CC$ as follows:
\begin{align*}
\CC\coloneqq \{\xi\in SH(Y,\xi^{(-)},\xi^{(+)})\cap L^\infty(Y,V) \mid \xi^{(-)}\leq\xi\leq\xi^{(+)}\}.
\end{align*}
\end{defi}
Clearly, $\xi^{(-)}$ and $\xi^{(+)}$ are contained in $\CC$. Then we prove the following:
\begin{lemm}\label{compact}{\it $\CC$ is a compact set for the $L^1$-topology, where the $L^1$-norm is with respect to the K\"ahler metric $g_X\left.\right|_Y$.
}
\end{lemm}
\begin{proof} Let $(\xi_{(k)})_{k\in\N}\in\CC^\N$ be a sequence. We denote by $\xi_{(k),j}$ the $j$-th component of $\xi_{(k)}$ for each $k\in\N$. From \cite[Theorem 1.46]{GZ1}, by extracting and relabeling, we can assume that $(\xi_j^{(-)}-\xi_{(k),j})_{k\in\N}$ converges to a subharmonic function $U_j$ in the $L^1_{loc}$-topology for all $j=1,\dots, r-1$. We set $\xi_j\coloneqq \xi_j^{(-)}-U_j$. We can also assume, by extracting and relabeling, that $(\xi_{(k),j}-\xi_j^{(+)})_{k\in\N}$ converges to a subharmonic function $U_j^\prime$ in the $L^1_{loc}$-topology for all $j=1,\dots, r-1$. We set $\xi_j^\prime\coloneqq U_j^\prime-\xi_j^{(+)}$. From the uniqueness of the convergence point, $\xi_j$ and $\xi_j^\prime$ coincide almost everywhere. We see that they coincide everywhere since the right-hand side of $\xi^{(-)}_j-\xi^{(+)}_j=\xi^{(-)}_j-\xi_{(k),j}+\xi_{(k),j}-\xi^{(+)}_j$ converges to $\xi_j^{(-)}-\xi_j+\xi_j^\prime-\xi^{(+)}_j$ as $k\to\infty$. We set $\xi\coloneqq (\xi_1,\dots, \xi_{r-1}, -(\xi_1+\cdots+\xi_{r-1}))$. Obviously, $\xi\in\CC$. Therefore all that remains is to prove $(\xi_{(k)})_{k\in\N}$ converges to $\xi$ in the $L^1$-topology. This can be done by the following observation for any compact subset $K$:
\begin{align*}
\int_Y|\xi-\xi_{(k)}|=\int_K |\xi-\xi_{(k)}|+\int_{Y\backslash K}|\xi-\xi_{(k)}|.
\end{align*}
The second term of the right-hand side can be arbitrarily small without depending on $k$, by only changing the selection of $K$ since there exists a compact exhaustion of $Y$ and $\xi_{(k)}$ and $\xi$ are bounded by $\xi^{(-)}$ and $\xi^{(+)}$.
\end{proof}

\begin{lemm}\label{xiprime}{\it For each $\xi\in\CC$, there uniquely exists a $\xi^\prime\in\CC$ that satisfies the following equation in the sense of the distribution:
\begin{align}
\Delta_{\omega_X}\xi^\prime+\sum_{j=1}^r4k_j^\prime e^{(v_j,\xi)}v_j=-2\inum\Lambda_{\omega_X}F_h. \label{linear}
\end{align}
Moreover, the unique solution $\xi^\prime$ to equation (\ref{linear}) is a $C^{1,\alpha}$-function for any $\alpha\in(0,1)$.
}
\end{lemm}
\begin{proof} 
We first construct a weak solution $\xi^\prime$ to equation (\ref{linear}). Similar to the proof of Lemma \ref{important lemma}, we assume that $F_h=0$ for simplicity. For each $j=1,\dots,r-1$, let $\xi_{j,+}^{\prime}$ be the unique $L^\infty$-weak solution of the following elliptic boundary value problem:
\begin{align*}
&\Delta_{\omega_X}\xi_{j,+}^{\prime}-4k_j^\prime e^{(v_j,\xi)}=0 \ \text{in the sense of the distribution}, \\
&-\xi_{j,+}^{\prime}\in SH(Y), \\ 
&\lim_{z\to\zeta}\xi_{j,+}^{\prime}(z)=\frac{1}{2}\eta_j(\zeta)\ \text{for all $\zeta\in\partial Y$},
\end{align*} 
where we denote by $SH(Y)$ the set of all subharmonic functions over $Y$. The existence and uniqueness of the $L^\infty$-solution to the above boundary value problem are guaranteed by \cite[p.150, Theorem 5.17]{GZ1}. We also denote by $\xi^\prime_{j,-}$ the unique $L^\infty$-solution of the following:
\begin{align*}
&\Delta_{\omega_X}\xi_{j,-}^{\prime}+4k_{j-1}^\prime e^{(v_{j-1},\xi)}=0 \ \text{in the sense of the distribution}, \\
&\xi_{j,-}^{\prime}\in SH(Y), \\ 
&\lim_{z\to\zeta}\xi_{j,-}^{\prime}(z)=\frac{1}{2}\eta_j(\zeta)\ \text{for all $\zeta\in\partial Y$}.
\end{align*}
We set $\xi^\prime_j\coloneqq \xi^\prime_{j,+}+\xi^\prime_{j,-}$ for each $j=1,\dots, r-1$. Then we define $\xi^\prime$ as $\xi^\prime\coloneqq (\xi^\prime_1,\dots,\xi^\prime_{r-1}, -(\xi_1^\prime+\cdots+\xi_{r-1}^\prime))$. From the construction, clearly, $\xi^\prime$ satisfies the elliptic equation (\ref{linear}) in the weak sense. From \cite[Theorem 6.2]{Ag1} (see also \cite[Theorem 1]{FF1}), it can be observed that $\xi_j^\prime \in W^{1,2}_{loc}(Y)$ for each $j=1,\dots, r$ (see \cite[Chapter 1]{Ag1} for the definition of the local Sobolev space). Therefore by \cite[Proposition 2.18]{FO1}, it follows that $\xi^\prime$ is a $V$-valued $C^{1,\alpha}$-function for any $\alpha\in (0,1)$. We next prove that $\xi^\prime$ is contained in $\CC$. We show that for each $j=1,\dots, r-1$, $\xi_j^{(-)}-\xi_j^\prime$ is a subharmonic function. From (\ref{iv0}), the following holds in the sense of the distribution:
\begin{align*}
\Delta_{\omega_X}(\xi_1^{(-)}-\xi_1^\prime)\leq& -4k_r^\prime e^{\xi_1^{(+)}-\xi_r^{(+)}} -2\Lambda_{\omega_X}F_{h_1} \\
&-4k_1^\prime e^{(v_1,\xi)}+4k_r^\prime e^{(v_r,\xi)}+2\inum\Lambda_{\omega_X}F_{h_j} \\
\leq & -4k_r^\prime (e^{\xi_1^{(+)}-\xi_r^{(+)}}-e^{\xi_1-\xi_r}) \\
\leq &\ 0,
\end{align*}
where the final inequality follows from the assumption $\xi_1\leq\xi_1^{(+)}$ and the fact $\xi_r^{(+)}=-(\xi_1^{(+)}+\cdots +\xi_{r-1}^{(+)})\leq -(\xi_1+\cdots+\xi_{r-1})=\xi_r$ which is immediately follows from the assumption $\xi_j\leq\xi_j^{(+)}$ for all $j=1,\dots, r-1$. Similarly, from (\ref{iv1}), the following holds for each $j=2,\dots, r-1$:
\begin{align*}
\Delta_{\omega_X}(\xi_j^{(-)}-\xi_j^\prime)\leq& -4k_{j-1}^\prime e^{\xi_j^{(+)}-\xi_{j-1}^{(-)}} -2\Lambda_{\omega_X}F_{h_j} \\
&-4k_j^\prime e^{(v_j,\xi)}+4k_{j-1}^\prime e^{(v_{j-1},\xi)}+2\inum\Lambda_{\omega_X}F_{h_j} \\
\leq & -4k_{j-1}^\prime (e^{\xi_j^{(+)}-\xi_{j-1}^{(-)}}-e^{\xi_j-\xi_{j-1}}) \\
\leq& \ 0.
\end{align*}
Since $\xi^\prime$ is a $V$-valued $C^{1,\alpha}$-function, $\xi_j^{(-)}-\xi_j^\prime$ is an upper semicontinuous function for each $j=1,\dots,r-1$. Therefore, combining the above inequality, it concludes that $\xi_j^{(-)}-\xi_j^\prime$ is a subharmonic function for each $j=1,\dots, r-1$. Then by applying the maximum principle \cite[Corollary 1.16]{GZ1} to $\xi_j^{(-)}-\xi_j^\prime$, we have $\xi_j^{(-)}- \xi_j^\prime\leq 0$ for each $j=1,\dots, r-1$. Therefore, in order to show that $\xi^\prime$ is contained in $\CC$, all that remains is to prove for each $j=1,\dots, r-1$, $\xi_j^\prime - \xi_j^{(+)}$ is a subharmonic function satisfying $\xi_j^\prime - \xi_j^{(+)}\leq 0$. From (\ref{iv2}), the following holds in the sense of the distribution:
\begin{align*}
\Delta_{\omega_X}(\xi_j^\prime-\xi_j^{(+)}) \leq & \ 4k_j^\prime e^{(v_j,\xi)}-4k_{j-1}^\prime e^{(v_{j-1},\xi)}-2\inum\Lambda_{\omega_X}F_{h_j} \\
& -4k_j^\prime e^{\xi_{j+1}^{(+)}-\xi_j^{(-)}} +2\Lambda_{\omega_X}F_{h_j} \\
\leq& -4k_j^\prime (e^{\xi_{j+1}^{(+)}-\xi_j^{(-)}}-e^{\xi_{j+1}-\xi_j}) \\
\leq& 0.
\end{align*}
Then by the same argument as above, it concludes that $\xi_j-\xi^{(+)}_j$ is a subharmonic function. Also, again, by the maximum principle, we have $\xi_j-\xi_j^{(+)}\leq 0$. Finally, we show the uniqueness of the weak solution to equation (\ref{linear}) which is contained in $\CC$. This follows from the maximum principle: Let $\xi^\prime=(\xi^\prime_1,\dots, \xi_r^\prime), \xi^{\prime\prime}=(\xi_1^{\prime\prime},\dots, \xi_r^{\prime\prime})\in\CC$ be weak solutions to equation (\ref{linear}). As we noted above, $\xi^\prime$ and $\xi^{\prime\prime}$ are $V$-valued $C^{1,\alpha}$-functions. Also, we have $\Delta_{\omega_X}(\xi^\prime-\xi^{\prime\prime})=0$. Since we have 
\begin{align}
\xi^{(-)}-\xi^{(+)}\leq \xi^\prime-\xi^{\prime\prime}\leq \xi^{(+)}-\xi^{(-)},
\end{align}
it holds that $\lim_{z\to\zeta}(\xi^\prime(z)-\xi^{\prime\prime}(z))=0$ for all $\zeta\in\partial Y$. Consequently, by the maximum principle (cf. \cite[Section 1.2.2]{GZ1}), we have $\xi_j^\prime=\xi_j^{\prime\prime}$ for all $j=1,\dots,r$. This establishes the desired result.
\end{proof}
We introduce the following notation:
\begin{defi}\label{S} We denote by $S:\CC\rightarrow \CC$ the map mapping a $\xi\in\CC$ to the unique $\xi^\prime\eqqcolon S(\xi)$ in Lemma \ref{xiprime}. 
\end{defi}
The following holds:
\begin{lemm}\label{L1}{\it The map $S$ defined in Definition \ref{S} is a continuous map in the $L^1$-topology.
}
\end{lemm}
\begin{proof} Let $(\xi_{(k)})_{k\in\N}\in\CC^\N$ be a sequence that converges to a $\xi=(\xi_1,\dots, \xi_r)\in\CC$ in the $L^1$-topology. We set $\xi_{(k)}^\prime\coloneqq S(\xi_{(k)})$ for each $k=1,2,\dots$. From Lemma \ref{compact}, by extracting and relabeling, we can assume that $(\xi_{(k)}^\prime)_{k\in\N}$ converges to a $\xi^\prime=(\xi^\prime_1,\dots,\xi^\prime_r)\in\CC$ in the $L^1$-topology. Consequently the assertion follows from the continuity of the Laplace operator with respect to the $L^1_{loc}$-topology.
\end{proof}
\begin{rem}\label{capacity} Unlike the case of the higher-dimensional Monge-Amp\`ere-operator (see \cite[pp. 154-155]{GZ1}), there is no need to show $(\xi_{(k)}^\prime)_{k\in\N}$, which has already been extracted and relabeled appropriately, converges to $\xi^\prime$ in capacity (see \cite[p.112, Definition 4.23]{GZ1}). However, if we additionally impose condition (\ref{iv3}) on $\xi^{(-)}$ and $\xi^{(+)}$, then it is not difficult to show that $(\xi_{(k)}^\prime)_{k\in\N}$ converges to $\xi^\prime$ in capacity, as shown in below: Let $\xi^\prime_{(k),j}$ be the $j$-th component of $\xi_{(k)}^\prime$ for each $j=1,\dots, r$ and each $k=1,2,\dots$. We show that for each $j=1,\dots, r-1$, $(\xi^\prime_{(k),j})_{k\in\N}$ converges to $\xi_j^\prime$ in capacity. For each Borel subset $E\subseteq Y$, we denote by $\Cap_Y(E)$ the capacity of $E$ (see \cite[p.108, Definition 4.16]{GZ1}). By using \cite[Lemma 5.18]{GZ1}, for each $\delta>0$, we have
\begin{align}
\Cap_Y(\{\xi_j^\prime-\xi^\prime_{(k),j}\geq 2\delta\})&=\Cap_Y(\{\xi^\prime_j-\xi_j^{(+)}-(\xi^\prime_{(k),j}-\xi_j^{(+)})\geq 2\delta\}) \notag \\
&\leq\delta^{-1}\int_{\{\xi_j^\prime-\xi_{(k),j}^\prime\geq \delta\}}dd^c(\xi^\prime_{(k),j}-\xi_j^{(+)}) \notag \\
&=(2\pi\delta)^{-1}\int_{\{\xi_j^\prime-\xi_{(k),j}^\prime\geq \delta\}}(-\Delta_{\omega_X})(\xi^\prime_{(k),j}-\xi_j^{(+)})\omega_X \notag \\
&\leq (2\pi\delta)^{-1}\int_{\{\xi_j^\prime-\xi_{(k),j}^\prime\geq \delta\}}(-\Delta_{\omega_X})(\xi^{(-)}_j-\xi_j^{(+)})\omega_X \notag \\
&\leq (2\pi\delta)^{-1}\int_{\{\xi_j^\prime-\xi_{(k),j}^\prime\geq \delta\}}2C\omega_X \notag \\
&\leq (2\pi\delta^2)^{-1}\int\max{\{\xi_j^\prime-\xi_{(k),j}^\prime, 0\}} \ 2C\omega_X \label{ComegaX}
\end{align}
where $C$ is a constant in (\ref{iv3}), and we have used the inequality $\Delta_{\omega_X}(\xi_j^{(-)}-\xi_{(k),j}^\prime)\leq 0$. (\ref{ComegaX}) converges to 0 as $k\to\infty$. Therefore we have $\lim_{k\to\infty}\Cap_Y(\{\xi_j^\prime-\xi^\prime_{(k),j}\geq 2\delta\})=0$. By swapping the roles of $\xi_j^\prime$ and $\xi^\prime_{(k),j}$, we can also show that $\lim_{k\to\infty}\Cap_Y(\{\xi_{(k),j}^\prime-\xi^\prime_j\geq 2\delta\})=0$. Therefore, $(\xi^\prime_{(k)})_{k\in\N}$ converges to $\xi^\prime$ in capacity. 
\end{rem}
Then we have the following:
\begin{lemm}\label{existence}{\it There exists a $V$-valued function $\xi=(\xi_1,\dots, \xi_r):Y\rightarrow V$ that satisfies (\ref{a}) and (\ref{b}) in Theorem \ref{main theorem 1}. 
}
\end{lemm}
\begin{proof} It is easy to observe that $\CC$ is a convex set. From Lemma \ref{compact}, $\CC$ is also a compact set with respect to the $ L^1$-topology. Then by the Schauder fixed point theorem (cf. \cite[p.143, Theorem 5.28]{Rud1}), the map $S$ defined in Definition \ref{S} has a fixed point $\xi$ since $S$ is a continuous map as shown in Lemma \ref{L1}. From Lemma \ref{xiprime}, the fixed point $\xi$ is a $V$-valued $C^{1,\alpha}$-function. The $V$-valued function $\xi$ solves equation (\ref{HEeq0}) in the sense of the distribution since $\xi$ is a fixed point of the map $S$. It can also be verified that the fixed point $\xi$ satisfies the boundary condition (\ref{b}) since we have $\xi^{(-)}\leq\xi\leq\xi^{(+)}$ and $\lim_{z\to\zeta}\xi^{(-)}(z)=\lim_{z\to\zeta}\xi^{(+)}(z)=\eta(\zeta)$ for all $\zeta\in\partial Y$. 
\end{proof}
The uniqueness of a solution $\xi$ in Theorem \ref{main theorem 1} follows from the maximum principle:
\begin{lemm}\label{uniqueness}{\it Let $\xi=(\xi_1,\dots, \xi_r)$ and $\xi^\prime=(\xi_1^\prime,\dots, \xi_r^\prime)$ be $V$-valued functions that satisfy (\ref{a}) and (\ref{b}) in Theorem \ref{main theorem 1}. Then we have $\xi=\xi^\prime$.
}
\end{lemm} 
\begin{proof} It follows from Proposition \ref{prop2} that $\log(\sum_{j=1}^re^{\xi_j-\xi_j^\prime})$ is a subharmonic function. By applying the maximum principle \cite[Corollary 1.16]{GZ1} to the subharmonic function $\log(\sum_{j=1}^re^{\xi_j-\xi_j^\prime})-\log r$, it concludes that $\log(\sum_{j=1}^re^{\xi_j-\xi_j^\prime})-\log r$ is a negative function. Since $\sum_{j=1}^r(\xi_j-\xi_j^\prime)=0$, it can be verified that $\log(\sum_{j=1}^re^{\xi_j-\xi_j^\prime})-\log r$ is also a positive function. Therefore $\xi_j=\xi_j^\prime$ for all $j=1,\dots, r$, and thus we have the result.
\end{proof}
From Lemma \ref{existence} and Lemma \ref{uniqueness}, we have Theorem \ref{main theorem 1}.

\medskip
\noindent
{\bf Acknowledgements.} I am grateful to Takahiro Aoi and Ryushi Goto for interesting discussions and stimulating conversations. I am also grateful to Tatsuya Tate for interesting discussions and stimulating conversations, and for helpful comments on the manuscript. I wish to express my gratitude to Hisashi Kasuya and Takashi Ono for their discussions on Simpson's papers. I am indebted to Toshiaki Yachimura for valuable discussions and for providing invaluable information on the regularity theory of elliptic PDEs. It was through his assistance that I was able to prove Lemma \ref{xiprime}. I would like to thank Yoshinori Hashimoto for interesting discussions, many stimulating conversations, and for answering my many fundamental questions about the Monge–Amp\`ere equation, potential theory, and related topics. I wish to thank Takuro Mochizuki for his guidance on cyclic Higgs bundles and harmonic bundles, as well as for answering my many questions about them when I was beginning my research on cyclic Higgs bundles.

\noindent
E-mail address 1: natsuo.miyatake.e8@tohoku.ac.jp

\noindent
E-mail address 2: natsuo.m.math@gmail.com \\

\noindent
Mathematical Science Center for Co-creative Society, Tohoku University, 468-1 Aramaki Azaaoba, Aoba-ku, Sendai 980-0845, Japan.
\end{document}